\documentclass{amsart}

\usepackage{graphicx}
\usepackage{epstopdf}
\usepackage{amsmath,amscd,color}

\newtheorem{thm}{Theorem}[section]
\newtheorem{cor}[thm]{Corollary}

\newtheorem{prop}[thm]{Proposition}

\theoremstyle{definition}
\newtheorem{defn}[thm]{Definition}
\newtheorem{example}[thm]{Example}
\theoremstyle{remark}
\newtheorem{rem}[thm]{Remark}
\numberwithin{equation}{section}

\newcommand{\R}{\mathbb R}
\newcommand{\Z}{\mathbb Z}
\newcommand{\C}{\mathbb C}

\newcommand{\F}{\mathcal{F}}

\newcommand{\ra}{\rightarrow}

\newcommand{\g}{\gamma}

\newcommand{\ds}{\displaystyle}
\newcommand{\OO}{{\mathcal O}}

\newcommand{\LL}{\mathcal{L}}

\newcommand{\pz}{PSL(2,\Z)}
\newcommand{\sz}{SL(2,\Z)}
\newcommand{\h}{\mathcal H}

\newcommand{\N}{\mathcal N}

\renewcommand{\Im}{\rm Im\,\,}
\renewcommand{\Re}{\rm Re\,\,}

\begin{document}

\title{Applications of $(a,b)$-continued fraction transformations}

\dedicatory{Dedicated to the memory of Dan Rudolph}

\author{Svetlana Katok}
\address{Department of Mathematics, The Pennsylvania State University,
University Park, PA 16802} 
\email{katok\_s@math.psu.edu}
\author{Ilie Ugarcovici}
\address{Department of Mathematical Sciences,
DePaul University, Chicago, IL 60614} 
\email{iugarcov@depaul.edu}
\thanks{The second author is partially supported by the NSF grant DMS-0703421}

\begin{abstract} We describe a general method of arithmetic coding of geodesics on the modular surface based on the study of one-dimensional Gauss-like maps associated to a two parameter family of continued fractions introduced in \cite{KU3}.  The finite rectangular structure of the attractors of the natural extension maps and the corresponding ``reduction theory" play an essential role. In special cases, when an $(a,b)$-expansion admits a so-called ``dual", the coding sequences are obtained by juxtaposition of the boundary expansions of the fixed points, and the set of coding sequences is a countable sofic shift. We also prove that the natural extension maps are Bernoulli shifts and compute  the density of the absolutely continuous invariant measure and the measure-theoretic entropy of the one-dimensional map. \end{abstract}

\maketitle
\section{Introduction and background}

In \cite{KU3}, the authors studied a  new two-parameter family of continued fraction transformations. These transformations can be defined using the standard generators $T(x)=x+1$, $S(x)=-1/x$ of the modular group  $\sz$ and considering $f_{a,b}:\bar\R\ra\bar\R$ given by
\begin{equation}\label{fab}
f_{a,b}(x)=\begin{cases}
x+1  &\text{ if }  x< a\\
-\displaystyle\frac{1}{x} &\text{ if } a\le x<b\\
x-1  &\text{ if } x\ge b\,.
\end{cases}
\end{equation}
Under the assumption that the parameters $(a,b)$ belong to the set
\[
\mathcal P=\{(a,b)\, |\, a\leq 0\leq b,\,b-a\geq 1,\,-ab\leq 1\}\,,
\]
one can introduce corresponding continued fraction algorithms by using the first return map of $f_{a,b}$ to the interval $[a,b)$. Equivalently, these so called {\em $(a, b)$-continued fractions} can be defined using 
a generalized integral part function: 
\begin{equation}
\lfloor x\rceil_{a,b}=\begin{cases}
\lfloor x-a \rfloor &\text{if } x<a\\
0 & \text{if } a\le x<b\\
\lceil x-b \rceil & \text{if } x\ge b\,,
\end{cases}
\end{equation}
where $\lfloor x\rfloor$ denotes the integer part of $x$ and $\lceil x\rceil=\lfloor x\rfloor+1$.

A starting point of the theory is the following result  \cite[Theorem 2.1]{KU3}:
if  $(a,b)\in \mathcal P$, then
any irrational number $x$
can be expressed uniquely as an infinite continued fraction of the form
\[
x=n_0-\cfrac{1}{n_1 -\cfrac{1}{n_2-\cfrac{1}{\ddots}}}= \lfloor n_0, n_1,\cdots\rceil_{a,b},\,\,(n_k\neq 0 \text{ for } k\ge 1),
\]
where $n_0=\lfloor x\rceil_{a,b}$, $x_1=-\frac1{x-n_0}$ and 
$n_{k+1}=\lfloor x_{k+1}\rceil_{a,b},\,x_{k+1}=- \frac1{x_k-n_k}$,
i.e. the sequence of partial fractions $r_k=\lfloor n_0,n_1,\dots ,n_k\rceil_{a,b}$ converges to $x$.

It is possible to construct $(a,b)$-continued fraction expansions for rational numbers, too. However, such expansions will terminate after finitely many steps if $b\ne 0$. If $b=0$, the expansions of rational numbers will end with a tail of $2$'s, since $0=\lfloor 1,2,2,\dots\rceil_{a,0}$. 

The above family of continued fraction transformations contains three classical examples: the case $a=-1$, $b=0$  described in \cite{Z, K3} gives the ``minus" (backward) continued fractions, the case $a=-1/2$, $b=1/2$ gives the ``closest-integer" continued fractions considered first by Hurwitz in \cite{Hurwitz1}, and the case $a=-1$, $b=1$ was presented in \cite{S1, KU2} in connection with a method of coding symbolically the geodesic flow on the modular surface following Artin's pioneering work \cite{Artin} and corresponds to the regular ``plus" continued fractions with alternating signs of the digits. 

The main object of study in \cite{KU3} is a two-dimensional realization of the {\em natural extension map} of $f_{a,b}$, $F_{a,b}:\bar\R^2\setminus\Delta\ra \bar\R^2\setminus\Delta$,  $\Delta=\{(x,y)\in \bar\R^2| x=y\}$, defined by
  \begin{equation}\label{Fab}
F_{a,b}(x,y)=\begin{cases}
(x+1,y+1)  &\text{ if } y< a\\
\Bigl(-\displaystyle\frac{1}{x},-\displaystyle\frac{1}{y}\Bigr) &\text{ if } a\le y<b\\
(x-1,y-1)  &\text{ if } y\ge b\,.
\end{cases}
\end{equation}
Here is the main result  of that paper:

\begin{thm}[\cite{KU3}] \label{main}
There exists an explicit one-dimensional Lebesgue measure zero, uncountable set $\mathcal E$ that lies on the diagonal boundary $b= a+ 1$ of $\mathcal P$ such that:
\begin{itemize}
\item[(1)] for all $(a,b)\in\mathcal P\setminus\mathcal E$ the map $F_{a,b}$ has an attractor $D_{a,b}=\cap_{n=0}^\infty F_{a,b}^n(\bar\R^2\setminus \Delta)$ on which $F_{a,b}$ is essentially bijective. 
\item[(2)] The set $D_{a,b}$ consists of two (or one, in degenerate cases) connected components each having {\em finite rectangular structure}, i.e. bounded by non-decreasing step-functions with a finite 
number of steps.  
\item[(3)] Almost every point $(x,y)$ of the plane ($x\ne y$) is mapped to $D_{a,b}$ after finitely many iterations of $F_{a,b}$.
\end{itemize}
\end{thm}

\begin{figure}[htb]
\includegraphics{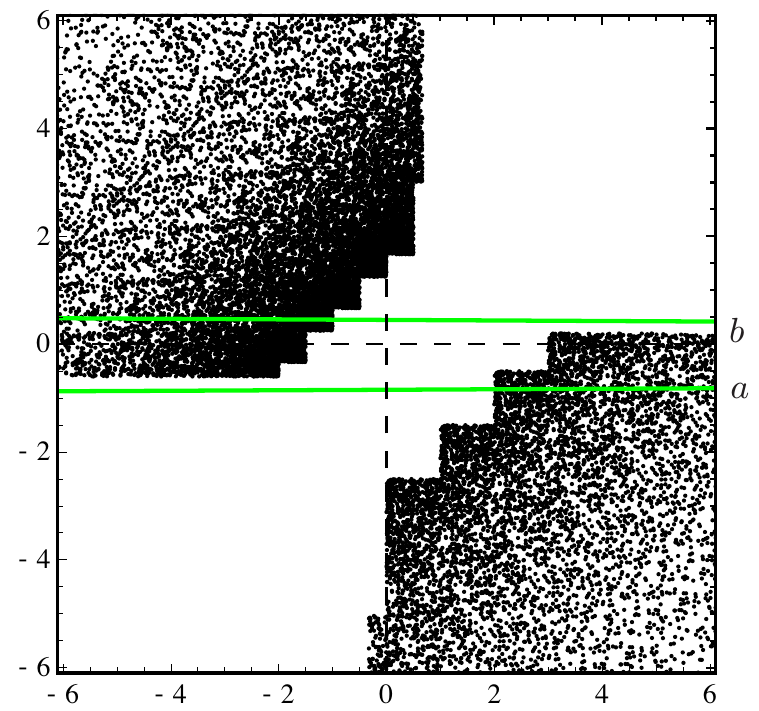}
\caption{Attracting domain $D_{a,b}$ for $a=-\frac{4}{5},\,b=\frac{2}{5}$}
\label{don-a}
\end{figure}

An essential role in the argument is played by the forward orbits associated to $a$ and $b$:  to $a$, the {\em upper orbit $\OO_u(a)$} (i.e. the orbit of $Sa$) and the {\em lower orbit $\OO_\ell(a)$} (i.e. the orbit of $Ta$), and to $b$, the {\em upper orbit $\OO_u(b)$} (i.e. the orbit of $T^{-1}b$) and  the {\em lower orbit $\OO_\ell(b)$} (i.e. the orbit of $Sb$). 
It was proved in \cite{KU3} that if $(a,b)\in\mathcal P\setminus\mathcal E$, then $f_{a,b}$ satisfies the {\em finiteness condition}, i.e. for both $a$ and $b$, their upper and lower orbits are either eventually periodic, or they satisfy the {\em cycle property}, i.e. they meet forming a cycle; more precisely, there exist
$k_1,m_1,k_2,m_2\geq 0$ s.t.
\[
f^{m_1}_{a,b}(Sa)=f^{k_1}_{a,b}(Ta)=c_a ,\,(\text{resp., } f^{m_2}_{a,b}(T^{-1}b)=f^{k_2}_{a,b}(Sb)=c_b),
\]
where $c_a$ and $c_b$ are the {\em ends of the cycles}.
If the products of transformations over the upper and lower sides of the cycle are equal, the cycle property is {\em strong}, otherwise, it is {\em weak}. In both cases  the set $\LL_{a,b}$ of the corresponding values is finite; ends of the cycles belong to the set $\LL_{a,b}$ if and only if they are equal to $0$, i.e. if the cycle is weak. The structure of the attractor $D_{a,b}$ is explicitly ``computed" from the finite set $\LL_{a,b}$.

The paper is organized as follows. In Section 2 we give some background information about geodesic flows and their
representations as  special flows over symbolic dynamical systems, and define the coding map.
In Section 3 we describe the reduction procedure for coding geodesics via $(a,b)$-continued fractions based on the study of the attractor of the associated natural extension map, define the corresponding cross-section set, and introduce the notion of {\em reduced geodesic}.
In Section 4 we prove that the first return map to the cross-section corresponds to a shift of the coding sequence (Theorem \ref{coding}) and, as a consequence, show that $(a,b)$-continued fractions satisfy the {\em Tail Property}, i.e. two $\sz$-equivalent real numbers have the same tails in their $(a,b)$-continued fraction expansions.
In Section 5 we introduce a notion of a {\em dual code} and show that if an $(a,b)$-expansion has a dual $(a',b')$-expansion, then the coding sequence of a reduced geodesic is obtained by juxtaposition of the  $(a,b)$-expansion of its attracting endpoint $w$ and the $(a',b')$-expansion of $1/u$, where $u$ is its repelling endpoint. We also prove that if the $(a,b)$-expansion admits a dual, then the set of admissible coding sequences is a sofic shift (Theorem \ref{soficshift}). 
In Section 6 we derive formulas for 
the density of the absolutely continuous invariant measure and the measure-theoretic entropy of the  one-dimensional Gauss-type maps and their natural extensions. We also prove that the natural extension maps are Bernoulli shifts. And finally, in Section 7 we apply results of \cite{KU3} to obtain explicit formulas for invariant measure for the one-dimensional maps for some regions of the parameter set $\mathcal P$.

\section{Geodesic flow on the modular surface and its representation as a special flow over a symbolic dynamical system}\label{codinggeodesicflow}

Let $\h=\{z=x+iy : y>0\}$ be the upper half-plane endowed with the
hyperbolic metric,
$\F=\{z\in\h : |z|\ge 1,\,\,|\Re\,
z|\le\frac{1}{2}\}$ be the standard fundamental region of the
modular group $\pz=\sz/\{\pm I\}$, and $M=\pz\backslash \h$ be the
modular surface. Let $S\h$ denote the unit tangent bundle of $\h$. We will use the coordinates $v=(z,\zeta)$ on $S\h$, where $z\in\h,\,\zeta\in\C,\,|\zeta|=\Im(z)$.
The quotient space $\pz\backslash S\h$ can be identified with the
unit tangent bundle of $M$, $SM$, although the structure of the
fibered bundle has singularities at the elliptic fixed points (see
\cite[\S 3.6]{K2} for details). Recall that geodesics in this model are half-circles or vertical half-rays.
The  geodesic flow $\{\tilde\varphi^t\}$ on $\h$ is defined as
an $\R$-action on the unit tangent bundle $S\h$ which moves a tangent vector along the geodesic defined by this vector with unit speed. The geodesic flow
$\{\tilde\varphi^t\}$ on $\h$ descents to the {\em geodesic flow}
$\{\varphi^t\}$ on the factor $M$ via the canonical projection 
\begin{equation}\label{pi}
\pi:S\h\to SM
\end{equation}
 of the unit tangent bundles. Geodesics on $M$ are orbits of the geodesic flow $\{\varphi^t\}$.

A {\em cross-section} $C$ for the geodesic flow is a subset of the
unit tangent bundle $SM$ visited by (almost) every geodesic
infinitely often both in the future and in the past. In other
words, every $v\in C$ defines an oriented geodesic $\g(v)$ on $M$
which will return to $C$ infinitely often. The ``ceiling" function
$g:C\to\R$ giving the {\em time of the first return} to $C$ is
defined as follows: if $v\in C$ and $t$ is the time of the first
return of $\g(v)$ to $C$, then $g(v)=t$. The map $R:C\to C$
defined by $R(v)=\varphi^{g(v)}(v)$ is called the {\em first
return map}. Thus $\{\varphi^t\}$ can be represented as a {\em
special flow} on the space
\[
{C}^{g}=\{(v,s) :  v\in C,\,\, 0\leq s\leq g(v)\},
\]
given by the formula $\varphi^t(v,s)=(v, s+t)$ with the
identification $(v,g(v))=(R(v),0)$.

\begin{figure}[thb]
\begin{center}    
\includegraphics{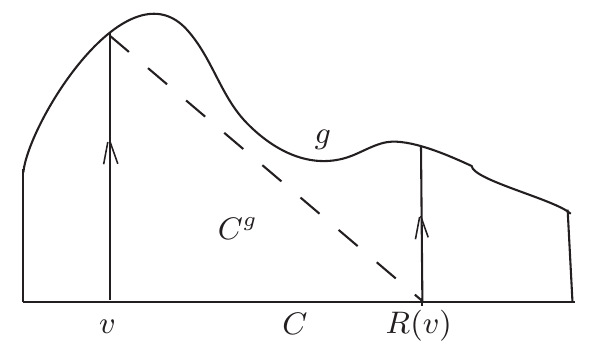}
\caption{Geodesic flow is a special flow}
\end{center}
\end{figure}

Let $\mathcal N$ be a finite or countable alphabet, $\N^{\Z}=\{x=\{n_i\}_{i\in \Z}\mid n_i\in\N\}$ be the space of all
bi-infinite sequences endowed with the Tikhonov (product) topology,
\[
\sigma:\N^\Z\to \N^\Z\text{ defined by  }(\sigma x)_i=n_{i+1}
\]
 be the left  shift map, and
$\Lambda\subset \N^\Z$ be a closed
$\sigma$-invariant  subset. Then $ (\Lambda,\sigma)$ is called a
{\em symbolic dynamical system}. There are some important classes of such dynamical systems. The space
$(\N^\Z, \sigma)$ is called the {\em full shift} (or the {\em topological Bernoulli shift}). If the space $\Lambda$
 is given by a
set of simple transition rules which can be described with the
help of a matrix consisting of zeros and ones, we say that $
(\Lambda,\sigma)$ is a {\em one-step topological Markov chain} or
simply a {\em topological Markov chain} (also called a {\em subshift of finite
type}). A factor of a topological Markov chain is called a {\em sofic shift}. (See  \cite[\S 1.9]{KH} for the definitions.)

\smallskip

In order to represent the geodesic flow as a special flow over a
symbolic dynamical system, one needs to choose an appropriate
cross-section $C$ and code it, i.e. to find an appropriate
symbolic dynamical system $ (\Lambda,\sigma)$ and a continuous
surjective map $\operatorname{Cod}: \Lambda\to C$
(in some cases the actual domain of $\operatorname{Cod}$ is
$\Lambda$ except a finite or countable set of excluded sequences) defined
such that the diagram
\[
\begin{CD}
\Lambda  @>{\operatorname{\sigma}}>>  \Lambda\\
@V{\operatorname{Cod}}VV      @VV{\operatorname{Cod}}V \\
C    @>{\operatorname{\mathit R}}>>C
\end{CD}
\]
is commutative. We can then talk about {\em coding sequences} for geodesics
defined up to a shift which corresponds to a return of the geodesic to the cross-section $C$. 
Notice that usually the coding map is not injective but only
finite-to-one (see e.g. \cite[\S 3.2 and \S 5]{Ad}).

There are two essentially different methods of coding geodesics on
surfaces of constant negative curvature. The geometric code, with respect to a given fundamental 
region, is obtained by a construction universal for all Fuchsian groups. The second method is specific for the modular group and is of arithmetic nature: 
it uses continued fraction expansions of the end points of the geodesic at infinity and 
a so-called reduction theory (see \cite{KU1,KU2} for the three classical  cases). Here we will describe a general method of arithmetic coding via $(a,b)$-continued fractions that is based on study of the attractor of the associated natural extension map. 
This approach, coupled with ideas of Bowen and Series \cite{BS}, may be useful for coding of geodesics on quotients by general Fuchsian groups.

\section{The reduction procedure} In what follows we
will denote the end points of geodesics on $\h$ by $u$ and $w$, and whenever we refer to such geodesics, we use $(u,w)$ as their coordinates on  $\bar\R^2$ ($u\neq w$). 

The reduction procedure for coding symbolically the geodesic flow on the modular surface via continued fraction expansions was presented for the three classical cases in \cite{KU2}; 
for a survey on symbolic dynamics of the geodesic flow see also \cite{KU1}. 
Here we  describe the reduction procedure for $(a,b)$-continued fractions and explain how it can be used for coding purposes. 

Let $\g$ be an arbitrary geodesic on $\h$ from $u$ to $w$ (irrational end points), and 
$w=\lfloor n_0,n_1,\dots\rceil_{a,b}$.
 We construct the sequence of
real pairs $\{(u_k,w_k)\}$ ($k\ge 0$) defined by
\begin{equation}\label{wk}
u_0=u,\,\,w_0=w \text{ and }w_{k+1}=ST^{-n_k}w_k\,,\quad u_{k+1}=ST^{-n_k}u_k\,.
\end{equation}
Each geodesic $\g_k$ from $u_k$ to $w_k$ is $\pz$-equivalent
to $\g$ by construction. It is convenient to describe this procedure using the {\em reduction map} that combines the appropriate iterate of the map $F_{a,b}$:
\[
R_{a,b}:\R^2\setminus\Delta\to \R^2\setminus\Delta
\]
given by the formula $R_{a,b}(u,w)=(ST^{-n}u,ST^{-n}w)$, where $n$ is the first digit in the $(a,b)$-expansion of $w$. Notice that $(u_k,w_k)=R^k_{a,b}(u,w)$.

\begin{defn}
 A geodesic in $\h$ from $u$ to $w$ is called {\em $(a,b)$-reduced} if $(u,w)\in \Lambda_{a,b}$, where
 \[
\Lambda_{a,b}=
F_{a,b}(D_{a,b}\cap\{a\leq w\leq b\})=S(D_{a,b}\cap\{a\leq w\leq b\}).
\]
\end{defn}

According to Theorem \ref{main},
for (almost) every geodesic $\g$ from $u$ to $w$ in $\mathcal H$, the above algorithm produces in finitely many steps an $(a,b)$-reduced geodesic $\pz$-equivalent to $\g$, and an application of this algorithm to an $(a,b)$-reduced geodesic produces another $(a,b)$-reduced geodesic. In other words,
there
exists a positive integer $\ell$ such that $R^\ell_{a,b}(u,w)\in\Lambda_{a,b}$ and 
\(
R_{a,b}:\Lambda_{a,b}\to \Lambda_{a,b}
\)
is bijective (with the exception of some segments of the boundary of $\Lambda_{a,b}$ and their images). 

Let $\g$ be a reduced geodesic with the repelling point $u\neq 0$ and the attracting point
\begin{equation}\label{ww}
w=\lfloor n_0,n_{1},\dots\rceil_{a,b}.
\end{equation}
Then, by successive applications of the map $R_{a,b}$, we obtain a sequence of real pairs $\{(u_k,w_k)\}$ ($k\ge 0$) (see \eqref{wk} above) such that
each geodesic $\g_k$ from $u_k$ to $w_k$ is $(a,b)$-reduced.
Using  the bijectivity of the map $R_{a,b}$, we extend the sequence (\ref{ww}) to the past to obtain a bi-infinite sequence of integers
\begin{equation}\label{codingseq}
\lfloor\g\rceil=\lfloor\dots,n_{-2},n_{-1},n_0,n_1,n_2,\dots\rceil,
\end{equation}
called the {\em coding sequence} of $\gamma$, as follows. There exists an integer $n_{-1}\neq 0$ and a real pair $(u_{-1},w_{-1})\in\Lambda_{a,b}$ such that $ST^{-n_{-1}}w_{-1}=w=w_0$ and $ST^{-n_{-1}}u_{-1}=u=u_0$. 
Notice that $\lfloor w_{-1}\rceil_{a,b}=n_{-1}$. By uniqueness of the $(a,b)$-expansion, we conclude
that $w_{-1}=\lfloor n_{-1},n_0,n_1,\dots\rceil_{a,b}$.
Continuing inductively, we define the sequence of integers $n_{-k}$
and
the real pairs $(u_{-k},w_{-k})\in\Lambda_{a,b}$ ($k\geq 2$), where 
\[
w_{-k}= \lfloor n_{-k},n_{-k+1}, n_{-k+2},\dots\rceil_{a,b}
\]
by $ST^{-n_{-k}}w_{-k}=w_{-(k-1)}$ and $ST^{-n_{-k}}u_{-k}=u_{-(k-1)}$. We also associate to $\g$ a bi-infinite sequence of $(a,b)$-reduced geodesics
\begin{equation}\label{seq-gamma}
(\dots, \g_{-2},\g_{-1}, \g_0,\g_1,\g_2,\dots),
\end{equation}
where $\g_k$ is the geodesic from $u_k$ to $w_k$. 
\begin{rem}\label{intermediate}
Notice that all ``intermediate" geodesics $T^{-s}\g_k$ ($1\leq s\le n_k$) obtained from $\g_k$ using the map $F_{a,b}$ are not $(a,b)$-reduced.

\end{rem}
\begin{prop} A formal minus continued fraction comprised from the digits of the ``past" of (\ref{codingseq}),
\[
n_{-1}-\cfrac{1}{n_{-2} -\cfrac{1}{n_{-3}-\cfrac{1}{\ddots}}}=(n_{-1},n_{-2},n_{-3},\dots )
\]
converges to $1/u$.
\end{prop}
\begin{proof} By \cite[Lemma 1.1]{KU},  it will be sufficient to check that $|n_{-k}|=1$ implies
$n_{-k}\cdot n_{-(k+1)}<0$, i.e. the digit $1$ must be followed by a negative integer and the digit $-1$ must be followed by a positive integer. We use the following properties of the set $\Lambda_{a,b}$ that can be derived from our knowledge of the shape of the set $D_{a,b}$ determined in \cite[Lemmas 5.6, 5.10, 5.11]{KU3}.
The upper part of  $\Lambda_{a,b}$ is contained in the region 
\begin{equation}\label{upper}
\begin{aligned}
&[-1,0]\times\Bigl[-\frac1{a},+\infty\Bigr]\cup[0,1]\times\Bigl[-\frac1{b-1},+\infty\Bigr]&\text{ if }b<1\\
&[-1,0]\times\Bigl[-\frac1{a},+\infty\Bigr]&\text{ if }b\ge 1.
\end{aligned}
\end{equation}
The lower part of $\Lambda_{a,b}$ is contained in the region 
\begin{equation}\label{lower}
\begin{aligned}
&[-1,0]\times\Bigl[-\infty,-\frac1{a+1}\Bigr]\cup[0,1]\times\Bigl[-\infty,-\frac1{b}\Bigr]&\text{ if }a> -1\\
&[0,1]\times\Bigl[-\infty,-\frac1{b}\Bigr]&\text{ if }a\le -1.
\end{aligned}
\end{equation}
Recall that $(u_{-(k+1)},w_{-(k+1)})=  (T^{n_{-(k+1)}}Su_{-k},T^{n_{-(k+1)}}Sw_{-k})$ for an appropriate integer $n_{-(k+1)}\neq 0$.
Suppose $n_{-k}=1$. Then $w_{-k}>0$. If $u_{-k}<0$,  then $Su_{-k}>0$ and $Sw_{-k}<0$,
and it takes a negative power of $T$ to bring it back to (the lower component of) $\Lambda_{a,b}$, i.e.
$n_{-(k+1)}<0$. The case $u_{-k}>0$, according to  (\ref{upper}), can only take place if $b\leq 1$.
In this case, $-1/(b-1)\le w_{-k}<b+1$, which is equivalent to $b>1$, a contradiction. Therefore $n_{-k}=1$ implies $n_{-(k+1)}<0$.
A similar argument shows that  $n_{-k}=-1$ implies $n_{-(k+1)}>0$. We conclude that the formal minus continued fraction converges. In order to prove that the limit is equal to $1/u$ we use the recursive definition of the digits $n_{-1},n_{-2},\dots$, to write
\[
\frac1{u}=n_{-1}-u_{-1}=n_{-1}-\cfrac{1}{n_{-2} -u_{-2}}=\cdots =(n_{-1},n_{-2},\dots ,n_{-k}-u_{-k})=\cdots,
\]
and the conclusion follows since the formal minus continued fraction converges.
\end{proof}

Let 
\[
C=\{z\in\h \mid |z|=1, \,-1\leq \Re z\leq 1\}
\] 
be the upper-half of the unit circle, and
\[
C^-=\{z\in\h \mid |z+1|=1, \,-\frac12\leq \Re z\leq 0\}
\] 
and 
\[
C^+=\{z\in\h \mid |z-1|=1,\,\,0\leq \Re z\leq \frac12\}
\] 
be the images of the two vertical boundary components of the fundamental region $\F$ under $S$ (see Figure \ref{fig-section}).
\begin{prop} \label{geod-reduced}  Every $(a,b)$-reduced geodesic either intersects $C$ or both curves $C^-$ and $C^+$.
\end{prop}
\begin{proof}
If $a,b$ are such that $-1\le a \le 0$ and $0\le b\le 1$, then by properties (\ref{upper}) and (\ref{lower}) of the set $\Lambda_{a,b}$, if $(u,w)\in \Lambda_{a,b}$, then $-1\leq u\leq 1$ and $w\geq -\frac1{a}$ or $w\leq -\frac1{b}$, and hence all $(a,b)$-reduced geodesics intersect $C$. For the case $b>1$ we have: if $-1<u<0$, then either $w>-\frac1{a}>b>1$ or $w<-\frac1{a+1}<-1$, i.e. the geodesic intersects $C$; if $0<u<1$, then (\ref{upper}) implies that  $w<-\frac1{b}<a<0$, thus the corresponding geodesic intersects $C$ if $w<-1$, and it intersects first $C^+$ and then $C^-$, if $-1<w<0$. Similarly, for the case $a<-1$ we have: 
if $0<u<1$, then either $w<-\frac1{b}<a<-1$
or $w>-\frac1{b-1}>1$, i.e. the geodesic intersects $C$; if $-1<u<0$, then
(\ref{lower}) implies that $w>-\frac1{a}>b>0$, therefore the corresponding geodesic intersects $C$ if $w>1$, and it intersects
first $C^-$ and then $C^+$ if $0<w<1$. \end{proof}

Based on Proposition \ref{geod-reduced} we introduce the notion of the {\em cross-section point}. It is either the intersection of a reduced geodesic $\g$ with $C$, or, if  $\g$ does not intersect $C$, its first intersection with $C^-\cup C^+$.
 
Now we can define a map
\[
\varphi: \Lambda_{a,b}\to S\h\,,\varphi(u,w)= (z,\zeta)
\]
where $z\in\h$ is the cross-section point on the geodesic $\g$ from $u$ to $w$, 
and $\zeta$ is the unit vector tangent to $\g$ at $z$.
The map is clearly injective. Composed with the canonical projection $\pi$ introduced in (\ref{pi}) we obtain a map 
\[
\pi\circ\varphi:\Lambda_{a,b}\to SM.
\] 

Let $C_{a,b}=\pi\circ\varphi(\Lambda_{a,b})\subset SM$. This set can be described as follows:
$C_{a,b}=P\cup Q_1\cup Q_2$, where
$P$ consists of the unit vectors based on the circular boundary of the fundamental region $\F$ 
pointing inward such that the corresponding geodesic $\g$ on the upper half-plane $\h$ is $(a,b)$-reduced,
$Q_1$ consists of the unit vectors based on the right vertical boundary of $\F$ pointing inward such that either $S\g$ or $TS\g$ is $(a,b)$-reduced (notice that they cannot both be reduced), and 
$Q_2$  consists of the unit vectors based on the left vertical boundary of $\F$ pointing inward such that either $S\g$ or $T^{-1}S\g$ is $(a,b)$-reduced (see Figure \ref{fig-section}). 
Then a.e. orbit of $\{\varphi^t\}$ returns to $C_{a,b}$, i.e. $C_{a,b}$ is a {\em cross-section} for $\{\varphi^t\}$, and
$\Lambda_{a,b}$ is a parametrization of $C_{a,b}$. The map $\pi\circ\varphi$ is injective, as follows from Remark \ref{intermediate}: only one of the geodesics $\g$, $S\g$, $T^{-1}S\g$, and $TS\g$ can be reduced.
\begin{figure}[htb]
\includegraphics{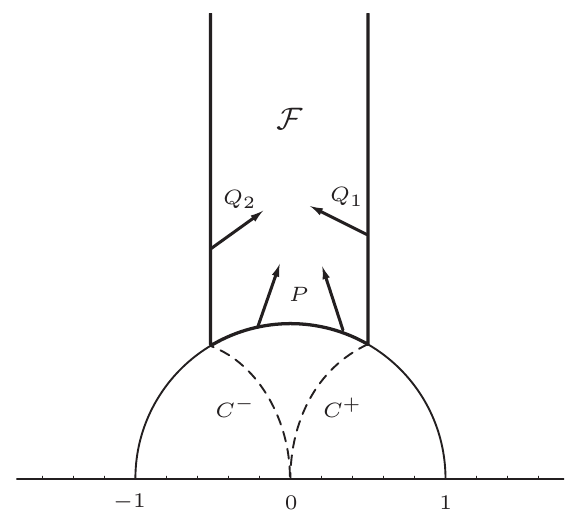}\hspace{20pt}
\includegraphics{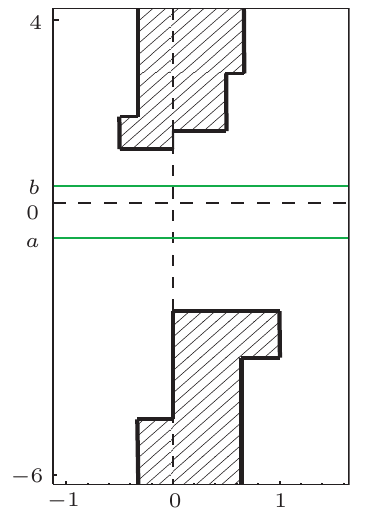}
\caption{The cross-section (left) and its $\Lambda_{a,b}$ parametrization (right)}
\label{fig-section}
\end{figure}

\section{Symbolic coding of the geodesic flow via $(a,b)$-continued fractions.}

If $\gamma$ is a geodesic on $\h$, we denote by $\bar\gamma$ the canonical projection of $\gamma$ on $M$. 
For a given geodesic on $M$  that can be reduced in finitely many steps, we can always choose its lift $\g$ to $\h$ to be $(a,b)$-reduced.

The following theorem provides the basis for coding geodesics on the modular surface using $(a,b)$-coding sequences.

\begin{thm} \label{coding}
Let $\g$ be an $(a,b)$-reduced geodesic on $\h$ and $\bar\g$ its projection to $M$. Then
\begin{enumerate}
\item each geodesic segment of $\bar\g$ between successive returns to the cross-section $C_{a,b}$ produces an $(a,b)$-reduced geodesic on $\h$, and each reduced geodesic $\sz$-equivalent to $\g$ is obtained this way;
\item the first return of $\bar\g$ to the cross-section $C_{a,b}$ corresponds to a left shift of the coding sequence of $\gamma$.
\end{enumerate}
\end{thm}
\begin{proof}
$(1)$ By lifting a geodesic segment on $M$ starting on $C_{a,b}$ to $\h$, we obtain a segment of a geodesic $\g$ on $\h$ that is reduced by the definition of the cross-section $C_{a,b}$. 
A coding sequence of $\g=\g_0$  that connects $u_0$ to $w_0=\lfloor n_0,n_1,\dots\rceil_{a,b}$,
\[
\lfloor\g_0\rceil=\lfloor\dots,n_{-2},n_{-1},n_0,n_1,n_2,\dots\rceil,
\]
 is obtained by extending the sequence of digits of $w_0$ to the past as explained in the previous section.

Let us assume that $w_0>0$, i.e. $n_0\geq 1$. The case $w_0<0$ can be treated similarly.
The geodesic $ST^{-n_0}\g_0=\g_1$ is reduced by Theorem \ref{main}. Let $z_0$ and $z_1$ be the cross-section points on $\g_0$ and $\g_1$, respectively. Then $z'_1=T^{n_0}Sz_1\in\g_0$; it is the intersection point of $\g_0$ with the circle $|z-n_0|=1$.
We will show that the geodesic segment of $\g_0$, $[z_0,z'_1]$ projected to $M$ is  the segment between two successive returns to the cross-section $C_{a,b}$. Since $ST^{-n_0}(z'_1)=z_1$ is the cross-section point on $\g_1$,  the geodesic segment $[z_0,z'_1]$ projected to $M$ is between two returns to $C_{a,b}$. Recall that a geodesic in $\F$ consists of countably many oriented geodesic segments between consecutive crossings of the boundary of $\F$ that are obtained by the canonical projection of $\g_0$ to $\F$.

If $z_0$ is the intersection of $\g_0$ with $C$, there are two possibilities. First, when $\g_0$ intersects $\F$ or $\g_0$ does not intersect $\F$ and   
$ST^{-1}\g_0$ exists $\F$ through it circular boundary, and, second, when $\g_0$ does not intersect $\F$ and $ST^{-1}\g_0$ exists $\F$ through it left vertical boundary. In the first case
the segments in $\F$ are represented by the intersection with $\F$ of the
following geodesics in $\h$: $T^{-1}\g_0,\,T^{-2}\g_0,\, \dots , \,T^{-n_0+1}\g_0$, either  $ST^{-n_0+1}\g_0$ or $T^{-n_0}\g_0$,  and either $\g_0$, or $ST^{-1}\g_0$. 

Suppose that
 for some intermediate point $z\in\g_0$, $z\in [z_0,z'_1]$ the unit vector tangent to $\g_0$ at $z$, $(z,\zeta)$ is projected to $C_{a,b}$. By tracing the geodesic $\g_0$ inside  $\F$, we see that $(z,\zeta)$ must be projected to $(\bar z,\bar\zeta)$ with $\bar z$ on the boundary of $\F$ and $\bar\zeta$ directed inward. Then the geodesic through $(\bar z,\bar\zeta)$
\begin{enumerate}
\item [(a)]enters $\F$ through its vertical boundary and exits it also through the vertical boundary,
\item [(b)]enters $\F$ through its vertical boundary and exits through its circular boundary, or
\item [(c)]enters $\F$ through its circular boundary and exits through its vertical boundary.
\end{enumerate}
The following assertions are implied by the analysis of the attractor $D_{a,b}$.
In case (a), $T^{-1}ST^{-s}\g_0$ is not reduced for $1\leq s<n_0$
since $s<n_0$, $T^{-s}w_0>b$, hence  $ST^{-s}w_0>-\frac1{b}$, i.e. $(ST^{-s}u_0,ST^{-s}w_0)\notin D_{a,b}$, therefore 
\[
(T^{-1}ST^{-s}u_0,T^{-1}ST^{-s}w_0)\notin \Lambda_{a,b}.
\]
In case (b), either the segment $T^{-n_0}\g_0$ exits through the circular boundary of $\F$, $ST^{-n_0}\g_0=\g_1$ is reduced, and we reached the point $z_1$ on the cross-section.
If the segment $T^{-n_0+1}\g_0$ intersects the circular boundary of $\F$,   $ST^{-n_0+1}\g_0$ is not reduced. In case (c), $ST^{-n_0+1}$ is not reduced.

In the second case the first digit of $w_0$, $n_0=2$. This is because $n_0=1$ would imply $b+1<w<-\frac1{b-1}$ which is impossible. Thus $ST^{-2}\g_0=\g_1$ is reduced. 
In this case the geodesic in $\F$ consists of the intersection with $\F$ of a single geodesic $ST^{-1}\g_0$ that enters $\F$ through its right vertical and leave it through its left vertical boundary, since $(TS)T(ST^{-1}\g_0)=ST^{-2}\g_0=\g_1$ is reduced.
In all cases the geodesic segment $[z_0,z'_1]$ projected to $M$ is between two consecutive returns to $C_{a,b}$.

If $z_0\notin C$, by Proposition \ref{geod-reduced}, since $w_0>0$, $z_0\in C^-$. 
Notice that this implies that $a<-1$ and $n_0=1$, and $\g_1=ST^{-1}\g_0$ is reduced. In this case the geodesic in $\F$ also consists of the intersection with $\F$ of a single geodesic $S\g_0$ that enters $\F$ through its right vertical and leave it through its left vertical boundary, since $(TS)T(S\g_0)=ST^{-1}\g_0=\g_1$ is reduced, and hence the geodesic segment $[z_0,z'_1]$ projected to $M$ is between two consecutive returns to $C_{a,b}$.
 Continuing this argument by induction in both positive and negative direction,  we obtain a bi-infinite sequence of points
 \[
(\dots , z_{-2},z_{-1},z_0,z_1,z_2,\dots),
\]
where $z_k$ is the cross-section point of the reduced geodesic $\g_k$ in the sequence of $\g_0$, that represents the sequence of all successive returns of the geodesic $\g_0$ in $M$ to the cross-section $C_{a,b}$. 

If $\tilde \g_0$ is a reduced geodesic in $\h$, $\sz$-equivalent to $\g_0$, then both project to the same geodesic on $M$. Therefore,  the cross-section point $\tilde z_0$ of $\bar \g_0$ projects on $C_{a,b}$ to a cross-section point $z_k$ of $\g_k$ for some $k$. This completes the proof of (1).

$(2)$ Since $\g_1=ST^{-n_0}\g_0$, $w_1=ST^{-n_0}w_0=\lfloor n_1,n_{2},\dots\rceil_{a,b}$. The first digit of the past is evidently $n_0$, and the remaining digits are the same as for $\g_0$. Thus (2) follows.
\end{proof}
The following corollary is immediate.
\begin{cor}\label{equivalentgeod}
If $\g'$ is $\sz$-equivalent to $\g$, and both geodesics can be reduced in finitely many steps, then the coding sequences of $\g$ and $\g'$ differ by a shift. 
\end{cor}
It implies a very important property of $(a,b)$-continued fractions that escapes a direct proof.
\begin{cor}\label{tail} {\bf (The Tail Property)} For almost every pair of real numbers  that are $\sz$-equivalent,  the ``tails" of their $(a,b)$-continued fraction expansions coincide.
\end{cor}

\begin{rem}
The set of exceptions in Corollary \ref{tail} is the same as the one described in Theorem \ref{main}(3).
\end{rem}
Thus we can talk about {\em coding sequences of geodesics on $M$}. 
To any geodesic $\g$ that can be reduced in finitely many steps we associate  the coding sequence (\ref{codingseq}) of a reduced geodesic $\sz$-equivalent to it. Corollary \ref{equivalentgeod} implies that this definition does not depend on the choice of a particular representative: sequences for equivalent reduced geodesics differ by a shift.

Let $X_{a,b}$ be the closure of the set of admissible sequences and $\sigma$ be the left shift map. The coding map $\operatorname{Cod}:X_{a,b}\rightarrow C_{a,b}$ is defined by
\begin{equation}\label{Cod}
\operatorname{Cod} (\lfloor\dots,n_{-2},n_{-1},n_0,n_1,\dots\rceil)=(1/(n_{-1},n_{-2},\dots ),\lfloor n_0,n_1,\dots\rceil_{a,b}). 
\end{equation}
This map is essentially bijective.

The symbolic system $(X_{a,b},\sigma)\subset (\N^\Z,\sigma)$ is defined on the
infinite alphabet $\N\subset\Z\setminus\{0\}$.
The product topology on $\N^\Z$ is induced by the distance
function
\[
d(x,x')=\frac 1{m}\,,
\]
where $x=(n_i), x'=(n'_i)\in \N^\Z$, and $m=\max \{k \mid
n_i=n'_i\text{ for } |i|\leq k\}$.
\begin{prop} \label{prop:cont}
The map $\operatorname{Cod}$ is continuous.
\end{prop}
\begin{proof}
If $d(x,x')<\frac 1{m}$, then the $(a,b)$-expansions of the
attracting end points $w(x)$ and $w(x')$ of the corresponding
geodesics given by (\ref{ww}) have  the same first $m$ digits.
Hence the first $m$ convergents of their $(a,b)$-expansions are the
same, and using the properties of $(a,b)$ continued fraction and the rate of convergence of\cite[Theorem 2.1]{KU3}
we obtain
$|w(x)-w(x')|<\frac {2}{m}$. Similarly, the first $m$ digits in the convergent formal minus continued fraction of
$\frac 1{u(x)}$ and $\frac 1{u(x')}$ are the same, and hence
$|u(x)-u(x')|<\frac {2|u(x)u'(x)|}{m}<\frac 2{m}$. Therefore the
geodesics are uniformly $\frac 2{m}$-close. But the tangent
vectors $v(x), v(x')\in C_{a,b}$ are determined by the intersection
of the corresponding geodesic with the unit circle or the curves $C^+$ and $C^-$. Hence, by
making $m$ large enough we can make   $v(x')$ as close to $v(x)$
as we wish.
\end{proof}

In conclusion, the geodesic flow becomes a special flow over a symbolic dynamical system $(X_{a,b},\sigma)$ on the infinite alphabet $\N\subset\Z\setminus\{0\}$. The ceiling function  $g_{a,b}(x)$ on $X_{a,b}$ coincides with the
time of the first return of the associated geodesic $\gamma(x)$ to
the cross-section $C_{a,b}$. One can establish an explicit formula for $g_{a,b}(x)$ as the function of the end points of the
corresponding geodesic $\gamma(x)$, $u(x)$, $w(x)$,
following the ideas explained in \cite{GK}. If $-1\le a \le 0$ and $0\le b\le 1$, then $g_{a,b}(x)$ is cohomologous to $2\log|w(x)|$; more precisely,
$$
g_{a,b}(x)=2\log |w(x)|+\log h(x)-\log h(\sigma x)
\text{ where }
h(x)=\frac{|w(x)-u(x)|\sqrt{w(x)^2-1}}{w(x)^2\sqrt{1-u(x)^2}}\,.
$$

\section{Dual codes}
We have seen that a coding sequence for a reduced geodesic from $u$ to $w$ (\ref{codingseq}) is comprised from the sequence of digits in $(a,b)$-expansion of $w$ and the ``past", an infinite sequence of non-zero integers, each digit of which depends on $w$ and $u$. In some special cases the ``past" only depends on $u$, and, in fact, it will coincide with the sequence of digits of $1/u$ by using a so-called  {\em dual expansion} to $(a,b)$.

Let $\psi(x,y)=(-y,-x)$ be the reflection of the plane about the line $y=-x$.
\begin{defn}
If  $\psi(D_{a,b})$ coincides with the attractor set $D_{a',b'}$ for some $(a',b')\in \mathcal P$, then the $(a',b')$-expansion is called the {\em dual} expansion to $(a,b)$. 
If $(a',b')=(a,b)$, 
then the 
$(a,b)$-expansion is called {\em self-dual}.
\end{defn}

\begin{example}
The classical situations of $(-1,0)$- and $(-1,1)$-expansions are self-dual. Two more sophisticated examples $(\frac{1-\sqrt{5}}2,\frac{3-\sqrt{5}}2)$ and $(-\frac3{8},\frac2{3})$, respectively, are shown in Figure \ref{fig-self}.
\begin{figure}[htb]
\begin{minipage}[b]{.49\textwidth}
  \begin{center}
    \includegraphics{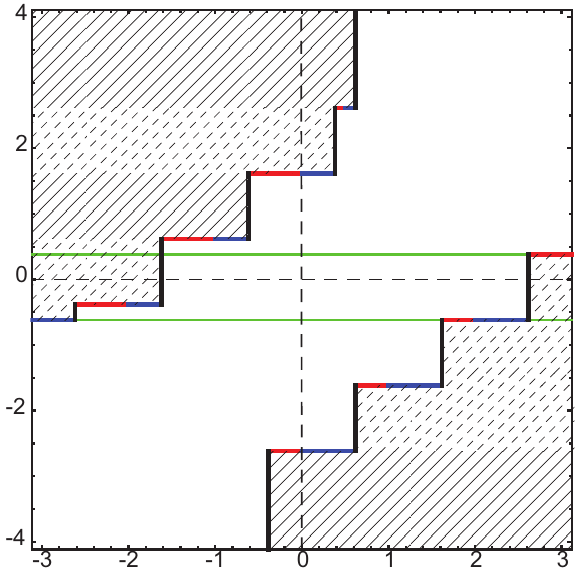}
    \center{$(a,b)=(\frac{1-\sqrt{5}}{2},\frac{3-\sqrt{5}}{2})$}
  \end{center}
  \end{minipage}
  \hfill
  \begin{minipage}[b]{.49\textwidth}\mbox{ }
    \begin{center}
\includegraphics{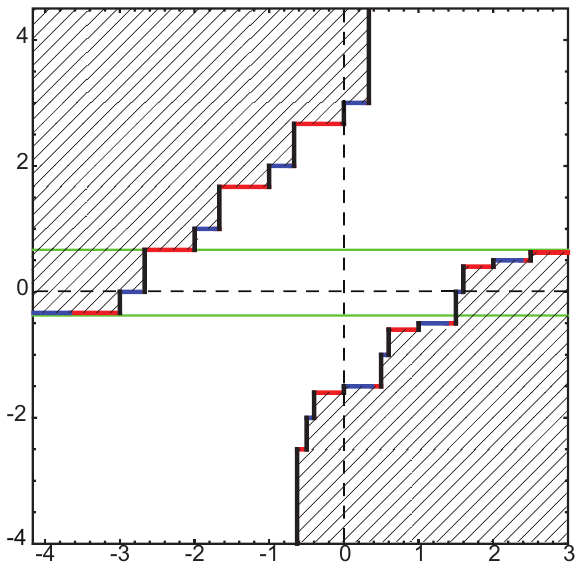}
\center{$(a,b)=(-\frac{3}{8},\frac{2}{3})$}
\end{center}
  \end{minipage}
  \hfill
\caption{Domains of self-dual expansions}
\label{fig-self}
\end{figure}
\end{example}

\begin{example}
The expansions $(-\frac1{n},1-\frac1{n})$, $n\ge 1$, satisfy a weak cycle property and have dual expansions that are periodic. A classical example in this series is the Hurwitz case $(-\frac1{2},\frac1{2})$ whose dual is $(\frac{1-\sqrt{5}}2,\frac{-1+\sqrt{5}}2)$  (see \cite{Hurwitz1,KU2}). Their domains are shown in Figure \ref{fig-dual}.

\begin{figure}[htbp]
\begin{minipage}[b]{.49\textwidth}
    \begin{center}
\includegraphics{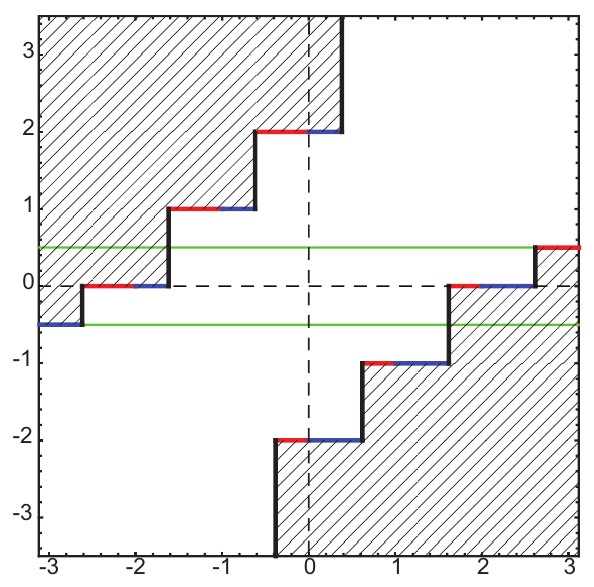}
\center{$(a,b)=(-\frac12,\frac12)$}
\end{center}
 \end{minipage}
  \hfill
  \begin{minipage}[b]{.49\textwidth}\mbox{ }
    \begin{center}
\includegraphics{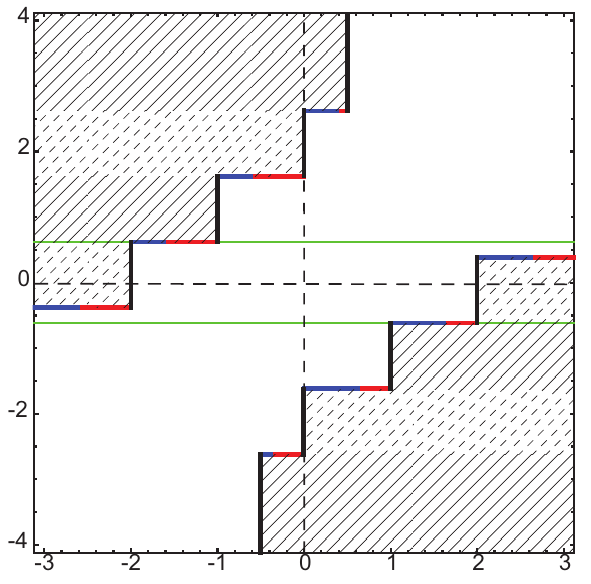}
\center{$(a',b')=(\frac{1-\sqrt{5}}{2},\frac{-1+\sqrt{5}}{2})$}
\end{center}
  \end{minipage}
  \hfill
\caption{Dual expansions}
\label{fig-dual}
\end{figure}
\end{example}

The following result gives equivalent characterizations for an expansion to admit a dual.

\begin{prop}\label{dual-properties} $ $
The following are equivalent:
\begin{itemize}
\item[(i)] the $(a,b)$-expansion has a dual;
\item[(ii)] the boundary of the lower part of the set $D_{a,b}$ does not have $y$-levels with $a<y<0$, and the  boundary of the upper part of the set $D_{a,b}$ does not have $y$-levels with $0<y<b$;
\item[(iii)] $a$ and $b$ do not have the strong cycle property.
\end{itemize}
\end{prop}
\begin{proof} 
If the $(a,b)$-expansion has a dual $(a',b')$-expansion, then the parameters $a',b'$ are obtained from the boundary of $D_{a,b}$ as follows: the right vertical boundary of the upper part of $D_{a,b}$ is the ray  $x=1-b'$,  and the left vertical boundary of the lower part of $D_{a,b}$
is the ray $x=-1-a'$. 
Now assume that (ii) does not hold. Then at least one of the parameters $a, b$ has the strong cycle property, and either the left boundary of the upper part of $\Lambda_{a,b}$ or the right  boundary of the lower part of $\Lambda_{a,b}$ is not a straight line. Assume the former. Then the reflection of $D_{a,b}$ with respect to the line $y=-x$ is not $D_{a',b'}$ since the map $F_{a',b'}$  is not bijective on it: the black rectangle in Figure \ref{noninj} belongs to it, but its image under $T^{-1}$, colored in grey, does not. Thus (i)$\Rightarrow$(ii).

\begin{figure}[htb]
\includegraphics{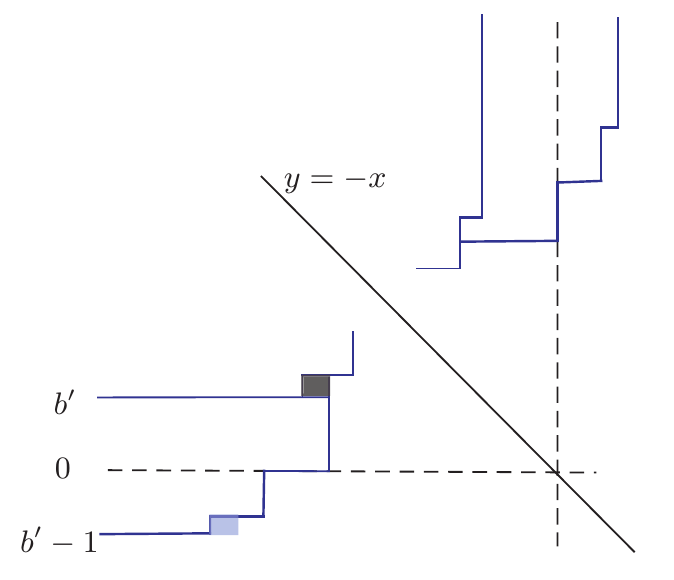}
\caption{Dual expansions and $D_{a,b}$}
\label{noninj}
\end{figure}

Conversely, let
the vertical line $x=1-b'$ be the right boundary
of the upper part of $D_{a,b}$ and
the vertical line $x=-1-a'$ be the left boundary of the lower part of $D_{a,b}$.
Let $[x_a,\infty]\times\{a\}$ be the intersection of $D_{a,b}$ with the horizontal line at the level $a$, and $[-\infty,x_b]\times\{b\}$  be the intersection of $D_{a,b}$ with the horizontal line at the level $b$. Then $a'=\frac1{x_b}$ and $b'=\frac1{x_a}$.
We also see that $1-b'=-\frac1{t}$, where $t=x_b$ or $t<x_b$ if $[t,x_b]\times\{0\}$ is a segment of the boundary of $D_{a,b}$. Then $-b'+1=-\frac1{t}\leq a'$, which implies $b'-a'\geq 1$. By Lemma 5.6 of \cite{KU3} $x_b\leq -1$ and $x_a\geq 1$, therefore 
\begin{equation}\label{De}
-1\leq a'\leq 0 \leq b'\leq 1,
\end{equation}
and
\begin{equation}\label{Lambda-strip}
\Lambda_{a,b}=D_{a,b}\cap\{(u,w)\in \bar\R^2 : -b'\leq u\leq -a'\}.
\end{equation}

We now show that $\psi(D_{a,b})=D_{a',b'}$ is the attractor for $F_{a',b'}$, where 
\begin{equation}\label{eq-dual}
F_{a',b'}=\psi\circ F^{-1}_{a,b}\circ\psi^{-1}.
\end{equation}
For $(u,w)\in D_{a',b'}$ with $a'<w<b'$ we have $\psi^{-1}(u,w)=(-w,-u)$ with $-b'<u<-a'$, so $\psi^{-1}(u,w)\in \Lambda_{a,b}$ by (\ref{Lambda-strip}), hence $F^{-1}_{a,b}(-w,-u)=(1/w,1/u)$, and $F_{a',b'}(u,w)=(-1/u,-1/w)$. 
For $(u,w)\in D_{a',b'}$ with $w>b'$ we have $\psi^{-1}(u,w)=(-w,-u)$ with $u<-b'$, so $F^{-1}_{a,b}(-w,-u)=(-w+1,-u+1)$, and $F_{a',b'}(u,w)=(u-1,w-1)$. 
Similarly, for $(u,w)\in D_{a',b'}$ with $w<a'$ we have $\psi^{-1}(u,w)=(-w,-u)$ with $u>-a'$, so $F^{-1}_{a,b}(-w,-u)=(-w-1,-u-1)$, and $F_{a',b'}(u,w)=(u+1,w+1)$. 
This proves that (ii)$\Rightarrow$(i).

Notice that  (ii) and (iii) are equivalent by Theorems 4.2 and 4.5 of \cite{KU3}.
\end{proof}

\begin{rem}\label{rem-dual}
Notice that if an $(a,b)$-expansion has a dual, then $-1\leq a\leq 0 \leq b\leq 1$. This follows from (\ref{De}) and the fact that the relation of duality is symmetric.
\end{rem}

\begin{thm} \label{coding-dual}If an $(a,b)$-expansion admits a dual expansion $(a',b')$, and $\g_0$ is an $(a,b)$-reduced geodesic, then its coding sequence 
\begin{equation}\label{codseq}
\lfloor \gamma_0\rceil
=\lfloor\dots,n_{-2},n_{-1},n_0,n_1,n_2,\dots\rceil,
\end{equation}
is obtained by juxtaposing the $(a,b)$-expansion of  $w_0=\lfloor n_0,n_1,n_2,\dots\rceil_{a,b}$ and the $(a',b')$-expansion of
$1/u_0=\lfloor n_{-1},n_{-2},\dots\rceil_{a',b'} $. This property is preserved under the left shift of the sequence.
\end{thm}
\begin{proof} We will show that the digits in the $(a',b')$-expansion of $1/u_0$ coincide with the digits of the ``past" of (\ref{codseq}). By \eqref{eq-dual}, the following diagram
\[
\begin{CD}
\Lambda_{a,b}  @>{\operatorname{S\psi}}>>  \Lambda_{a',b'}\\
@V{\operatorname{R^{-1}_{a,b}}}VV      @VV{\operatorname{R_{a',b'}}}V \\
\Lambda_{a,b}      @>{\operatorname{S\psi}}>> \Lambda_{a',b'}
\end{CD}
\]
is commutative. The pair $(u_0,w_0)\in \Lambda_{a,b}$, therefore $(Su_0,Sw_0)\in S\Lambda_{a,b}\subset D_{a,b}$, and $(1/w_0,1/u_0)\in\Lambda_{a',b'}$.
The first digit of the $(a',b')$-expansion of $1/u_0$ is $n_{-1}$, so
\[
R_{a',b'}(1/w_0,1/u_0)=(ST^{-n_{-1}}(1/w_0),ST^{-n_{-1}}(1/u_0))
\]
 maps $\Lambda_{a',b'}$ to itself. Then 
 \[
 (u_{-1},w_{-1}):=R^{-1}_{a,b}(u_0,w_0)=(T^{n_{-1}}Su_0,T^{n_{-1}}Sw_0)\in\Lambda_{a,b}
 \]
 and $(ST^{-n_{-1}}u_{-1},ST^{-n_{-1}}w_{-1})=(u_0,w_0)$. Also  
 $w_{-1}=\lfloor n_{-1},n_0,n_1,\dots\rceil_{a,b}$, and $ST^{-n_{-1}}(1/u_0)=1/u_{-1}=\lfloor n_{-2},\dots\rceil_{a',b'} $.
 
Continuing by induction, one proves  that all digits of the ``past" of the sequence (\ref{codseq}) are the digits of the $(a',b')$-expansion of $1/u_0$.

In order to see what happens under a left shift, we reverse the diagram to obtain:
 \[
\begin{CD}
\Lambda_{a,b}  @>{\operatorname{S\psi}}>>  \Lambda_{a',b'}\\
@V{\operatorname{R_{a,b}}}VV      @VV{\operatorname{R^{-1}_{a',b'}}}V \\
\Lambda_{a,b}      @>{\operatorname{S\psi}}>> \Lambda_{a',b'}
\end{CD}
\]
 Since the first digit of $(a,b)$-expansion of $w_0$ is $n_0$,
 \[
 R_{a,b}(u_0,w_0)=(ST^{-n_0}u_0,ST^{-n_0}w_0)
 \]
 maps $\Lambda_{a,b}$ to itself. Then $(u_1,w_1):=(ST^{-n_0}u_0,ST^{-n_0}w_0)$ and $w_1=\lfloor n_1,n_2,\dots\rceil_{a,b}$. Also 
 \[
 (1/w_1,1/u_1)=R^{-1}_{a',b'}(1/w_0,1/u_0)=(T^n_0S(1/w_0),T^n_0S(1/u_0)),
 \]
 hence $1/u_1=\lfloor n_0,n_{-1},n_{-2},\dots\rceil_{a',b'}$.
 \end{proof}
\begin{rem}
Under conditions of Theorem \ref{coding-dual}, if $\gamma_0$  projects to a closed geodesic on $M$, then its coding sequence is periodic, and $w_0=\lfloor \overline{n_0,n_1,\dots,n_m}\rceil_{a,b}$, $1/{u_0}=\lfloor \overline{n_m,\dots, n_1,n_0}\rceil_{a',b'}$.
\end{rem}

\begin{thm} \label{soficshift}
If an $(a,b)$-expansion admits a dual expansion, then the symbolic space $(X_{a,b},\sigma)$ is a sofic shift.
\end{thm}
\begin{proof} The ``natural" (topological) partition of the set $\Lambda_{a,b}$ related to the alphabet $\N$ is $\Lambda_{a,b}=\displaystyle\cup_{n\in\N}\Lambda_n$, where $\Lambda_n$ are labeled by the  symbols of the alphabet $\N$ and are defined
by the following condition: $\Lambda_{n}=\{(u,w)\in \Lambda_{a,b}\mid
n_{0}(u,w)=n_0(w)=n\}$. In order to prove that the space $(X_{a,b},\sigma)$ is sofic one needs to find a topological Markov chain $(M_{a.b},\tau)$ and a surjective continuous map $h:M_{a,b}\to X_{a,b}$ such that $h\circ\tau=\sigma\circ h$.

Notice that the elements $\Lambda_n$ are rectangles for large $n$; in fact, at most two elements in the upper part and at most two elements in the lower part of $\Lambda_{a,b}$ are incomplete rectangles (see  Figure \ref{partition}).

\begin{figure}[htb]
\begin{center}
\includegraphics{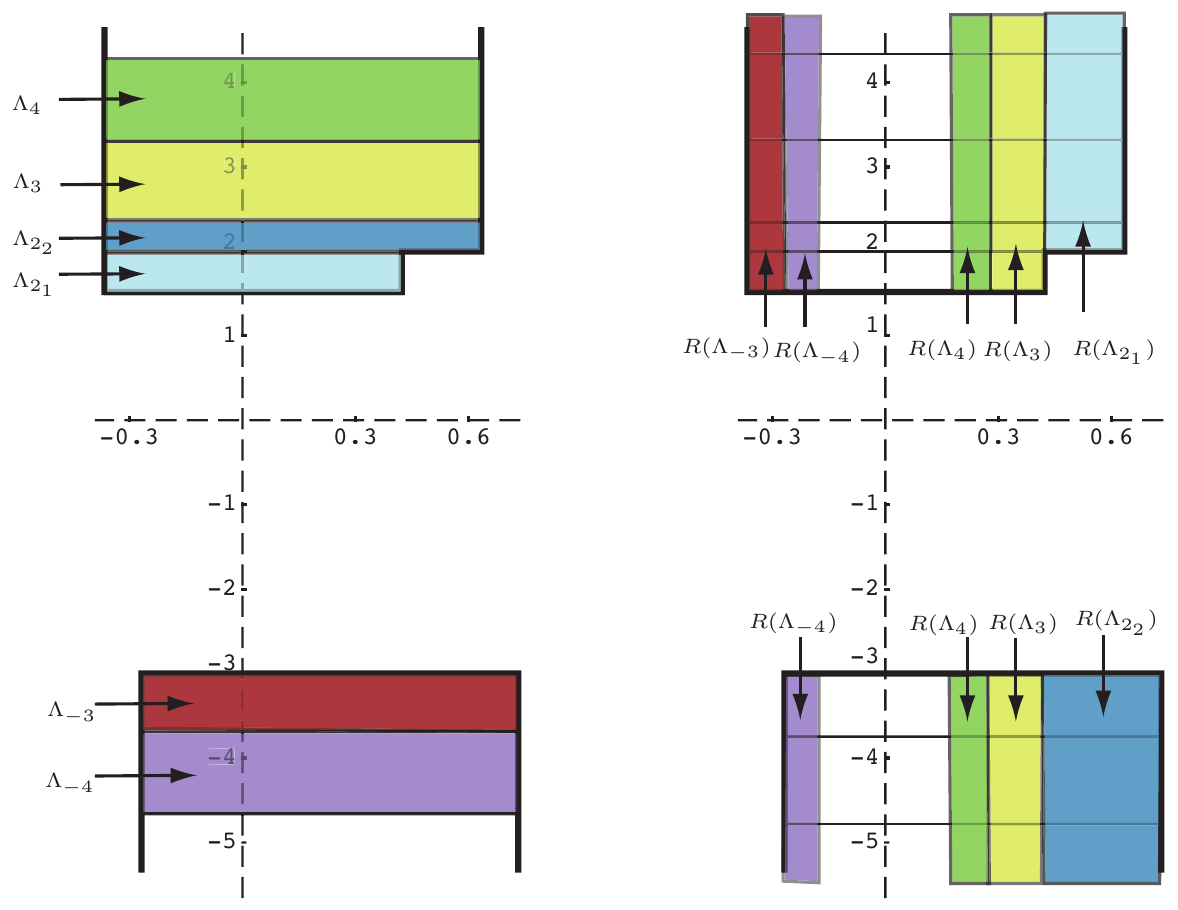}
\caption{The partition of $\Lambda_{a,b}$ and its image through $R_{a,b}$.}
\label{partition}
\end{center}
\end{figure}

Since $\Lambda_{a,b}$ has finite rectangular structure, we can sub-divide horizontally these incomplete rectangles into rectangles, and extend the alphabet $\N$ by adding subscripts to the corresponding elements of $\N$. For example, if $\Lambda_2$ is subdivided into two rectangles, $\Lambda_2=\displaystyle\cup_{i=1}^2\Lambda_{2_i}$, the ``digit" $2$ will give rise to two digits, $2_1,2_2$ in the extended alphabet $\N'$ (see Figure \ref{partition}). We denote the new partition of $\Lambda_{a,b}$ by $\displaystyle\cup_{n\in\N'}M_n$. Notice that it consists of rectangles with horizontal and vertical sides.
Since the first return $R$ to $\Lambda_{a,b}$
corresponds to the left shift of the coding sequence $x$
associated to the geodesic $(u,w)$, we see that $x=\{n_k\}_{-\infty}^{\infty}$, where $n_k$ is defined by $R^k(u,w)\in \Lambda_{n_k}$.
Now we define the symbolic space $M_{a,b}$ as follows: to each sequence $x\in X_{a,b}$ we associate a geodesic $(u,w)$ by (\ref{Cod}), and define a new coding sequence $y=\{m_k\}_{-\infty}^{\infty}$, where $m_k$ is defined by
$R^k(u,w)\in M_{m_k}$, and $\tau$ is the left shift.

We will prove that $(M_{a,b},\tau)$ is a topological Markov chain. For this, in accordance to \cite[Theorem 7.9]{Ad}, it is sufficient to prove that for any pair of distinct symbol $n,m\in\N'$,
$R(M_n)$ and $M_m$ either do not intersect, or intersect ``transversally" i.e. their intersection is a rectangle
with two horizontal sides belonging to the horizontal boundary of $M_m$ and two vertical sides belonging to the vertical boundary of $R(M_n)$. Let us recall that $-1\leq a\leq 0 \leq b\leq 1$ (see Remark \ref{rem-dual}).
Therefore, if $M_n=\Lambda_n$ is a complete rectangle, it is, in fact,  a $1\times 1$ square, and its image under $R$ is an infinite vertical rectangle intersecting all $M_m$ transversally. If $M_n$ is obtained by subdivision of some $\Lambda_k$ and belongs to the lower part of $\Lambda_{a,b}$, its horizontal boundaries are the levels of the step-function defining the lower component of $D_{a,b}$, and by Proposition \ref{dual-properties}, since the lower boundary of $D_{a,b}$ does not have $y$-levels with $a<y<0$,
its image is a vertical rectangle intersecting only the lower component of $D_{a,b}$ whose horizontal boundaries are the levels of the step-function defining the lower component of $D_{a,b}$. Therefore, all possible intersections with $M_m$ are transversal. A  similar argument applies to the case when $M_n$ belongs to the upper part of $\Lambda_{a,b}$. The map $h:M_{a,b}\to X_{a,b}$ is obviously continuous, surjective, and, in addition,  $h\circ\tau=\sigma\circ h$.
\end{proof}

\section{Invariant measures and ergodic properties}\label{s:9}

Based on the finite rectangular geometric structure of the domain $D_{a,b}$ and the connections with the geodesic flow on the modular surface, we study some of the measure-theoretic properties of the Gauss-type map $\hat f_{a,b}:[a,b)\rightarrow [a,b)$,
\begin{equation}\label{1dGauss}
\hat f_{a,b}(x)=-\frac{1}{x}-\left\lfloor -\frac{1}{x}\right\rceil_{a,b}\,,\quad  \hat f_{a,b}(0)=0\,.
\end{equation}
Notice that the associated natural extension map $\hat F_{a,b}$ 
\begin{equation}\label{2dGauss}
\hat F _{a,b}(x,y)=\left(\hat f_{a,b}(x),-\frac{1}{y-\lfloor -1/x\rceil_{a,b}}\right)
\end{equation}
is obtained from the map $F_{a,b}$ induced on the set $\Lambda_{a,b}$ by the change of coordinates 
\begin{equation}\label{cchange}
x=-1/w,\,\,y=u
\end{equation}
 (or, equivalently, on the set $D_{a,b} \cap \{(u,w)| a\le w <b\}$ by the change of coordinates $x=w, y=-1/u$). Therefore the domain $\hat \Lambda_{a,b}$ of $\hat F _{a,b}$
 is easily identified knowing $\Lambda_{a,b}$ and may be considered as its ``compactification". 

Many of the measure-theoretic properties of $\hat f_{a,b}$ and $\hat F_{a,b}$ (existence of an absolutely continuous invariant measure, ergodicity) follow from the fact that the geodesic flow $\varphi^t$ on the modular surface $M$ can be represented as a special
flow 
$(R_{a,b},\Lambda_{a,b},g_{a,b})$ on the space 
\[
\Lambda_{a,b}^{g_{a,b}}=\{(u,w,t) : (u,w)\in \Lambda_{a,b},\,\,0\leq t\leq g_{a,b}(u,w)\}
\]
(see Section 2). We recall that $R_{a,b}=F_{a,b}|_{\Lambda_{a,b}}$ and $g_{a,b}$ is the ceiling function (the time of the first return to the cross-section $C_{a,b}$) parametrized by $(u,w)\in \Lambda_{a,b}$.

We start with the fact that the geodesic flow $\{\varphi^t\}$ preserves the smooth (Liouville) measure $dm=\displaystyle\frac{dudwdt}{(w-u)^2}$ (see, e.g., \cite{AF1}), hence $R_{a,b}$ preserves the absolutely continuous measure $d\rho=\displaystyle\frac{dudw}{(w-u)^2}$. Using the change of coordinates \eqref{cchange}, the map $\hat F_{a,b}$ preserves the absolutely continuous measure $d\nu=\displaystyle\frac{dxdy}{(1+xy)^2}$.

The set $\Lambda_{a,b}$ has finite  measure $d\rho$
if $a\ne 0$ and $b\ne 0$, since it is uniformly bounded away from the line $\Delta=\{(u,w)\: :u=w\} \subset \R^2$ (see relations \eqref{upper} and \eqref{lower}). In this situation, we can normalize the measure $d\rho$ to obtain the smooth probability measure 
\begin{equation}\label{drho}
d\rho_{a,b}=\frac{d\rho}{K_{a,b}}=\frac{dudw}{K_{a,b}(w-u)^2}
\end{equation}
where $K_{a,b}=\rho(\Lambda_{a,b})$.
Similarly, if $a\ne 0$ and $b\ne 0$, the map $\hat F_{a,b}$ preserves the smooth probability measure 
\begin{equation}\label{nu}
d\nu_{a,b}=\frac{dxdy}{K_{a,b}(1+xy)^2}
\end{equation}
and $K_{a,b}=\rho(\Lambda_{a,b})=\nu(\hat\Lambda_{a,b})$. 

Returning to the Gauss-type map, $\hat f_{a,b}$, one can obtain explicitly a Lebesgue equivalent invariant probability measure $\mu_{a,b}$ by projecting the measure $\nu_{a,b}$ onto the $x$-coordinate (push-forward); this is equivalent to integrating $\nu_{a,b}$ over $\hat \Lambda_{a,b}$ with respect to the $y$-coordinate as explained in \cite{AF3}. 

We can immediately conclude that the systems $(\hat F_{a,b},\nu_{a,b})$ and $(\hat f_{a,b},\mu_{a,b})$ are ergodic from the fact that the geodesic flow $\{\varphi^t\}$ is ergodic with respect to $dm$. By using some well-known results about one dimensional maps that are piecewise monotone and expanding, and the implications for their natural extension maps, we can establish stronger measure-theoretic properties: $(\hat f_{a,b},\mu_{a,b})$ is exact, and $(\hat F_{a,b},\nu_{a,b})$ is a Bernoulli shift. Here we follow the presentation from \cite{Zw} based on \cite{W, R}.

\begin{thm}\label{Bern}
For any $a\ne 0$ and $b\ne 0$, the system $(\hat f_{a,b},\mu_{a,b})$ is exact and its natural extension $(\hat F_{a,b},\nu_{a,b})$ is a Bernoulli shift.
\end{thm}
\begin{proof}
Let us consider first the case $-1<a<0<b<1$. The interval $(a,b)$ admits a countable partition  $\xi=\{X_i\}_{i\in \Z\setminus\{0\}}$ of open intervals and the map $\hat f_{a,b}$ satisfies conditions (A), (F), (U) listed in \cite{Zw}. Condition (A) is Adler's distortion estimate: 

\smallskip

\( 
\qquad\qquad (A):  \quad \hat f''_{a,b}/(\hat f_{a,b}')^2 \text{ is bounded on } \displaystyle X=\cup_{i\in \Z\setminus\{0\}}  X_i,
\)

\medskip

\noindent condition (F) requires the finite image property of the partition $\xi$, 

\smallskip
\(
\qquad\qquad (F): \quad \hat f_{a,b}(\xi)=\{\hat f_{a,b}(X_i)\}_{i\in \Z\setminus\{0\}}\text{ is finite},
\)

\medskip

\noindent while condition (U) is a uniformly expanding condition

\smallskip

\( 
\qquad\qquad (U): \quad |\hat f'_{a,b}|\ge \tau >1 \text{ on } X.
\)

\medskip

\noindent Let $m\ge 0$ and $n\ge 0$ be such that $ a-m\le -1/b<a-m-1$ and $b+n\le -1/a<b+n+1$. Consider the open intervals

\[
X_1=\bigg(-\frac{1}{a-m-1}\,,\,b\bigg),\; X_i=\bigg(-\frac{1}{a-m-i}\,,\,-\frac{1}{a-m-i+1}\bigg) \text{ for } i\ge 2
\]
and 
\[
X_{-1}=\bigg(a\,,-\frac{1}{b+n+1}\bigg),\; X_{-i}=\bigg(-\frac{1}{b+n+i-1}\,,\,-\frac{1}{b+n+i}\bigg) \text{ for } i\ge 2.
\]

The map $\hat f_{a,b}$ satisfies conditions (A), (F), (U) with respect to the partition $\xi=\{X_i\}_{i\in \Z\setminus\{0\}}$. Indeed, 
$|\hat f''_{a,b}/(\hat f'_{a,b})^2|\le 2$ on $X$,  the collection of images $\hat f_{a,b}(\xi)$ consists of four sets $\hat f_{a,b}(X_1)$, $\hat f_{a,b}(X_{-1})$, $(b-1,b)$, $(a,a+1)$, and $|\hat f'_{a,b}|\ge \min\{\frac{1}{a^2},\frac{1}{b^2}\}>1$ on $X$. 
Zweim\"uller \cite{Zw} showed that any one-dimensional map for which conditions (A), (F), (U)  hold is exact and satisfies Rychlik's conditions described in \cite{R}, hence its natural extension map is Bernoulli.

We analyze now the case $b\ge 1$. Let $K>0$ be the smallest integer such that $b(a+1)^K<1$. We  will show
that there exists $\gamma>1$ such that, for every $x\in \bigcap_{i=0}^K \hat f_{a,b}^{-i}(X)$, some iterate $\hat f^n_{a,b}(x)$ with $n\le K+1$ is expanding, i.e. $|(\hat f_{a,b}^n)'(x)|\ge \gamma$. (For the rest of the proof, we simplify the notations and let $\hat f$ denote the map $\hat f_{a,b}$.)
Notice that if $x\in\bigcap_{i=0}^{n-1} \hat f^{-i}(X)$, then $\hat f^n$ is differentiable at $x$ and 
$$\frac{d}{dx}\hat f^n(x)
=\frac{1}{(x\hat f(x)\cdots\hat f^{n-1}(x))^2}\,.$$

\noindent Assume that $ab>-1$. We look at the following cases:
\begin{itemize}
\item[(i)] If $a<x<0$, then $b-1\le \hat f(x)\le b$, and  $|x\hat f(x)|\le |ab|<1$.
\item[(ii)] If $0<x<b$, then $a\le \hat f(x)\le a+1$. Let $K$ be such that $b(a+1)^K<1$. Then either there exists $1\le n\le K$ such that $0<\hat f^i(x)<a+1$ for $i=1,2,\dots,n-1$ and $a<\hat f^n(x)<0$, or $0<\hat f^i(x)<a+1$ for $i=1,2,\dots,K$. In the former case we have that
\begin{equation}\label{eq-rel1}
|x\hat f(x)\cdots\hat f^n(x)|\le |ab(a+1)^{n-1}|<1\,,
\end{equation}
while in the latter case
\begin{equation}\label{eq-rel2}
|x\hat f(x)\cdots\hat f^K(x)|\le |b(a+1)^K|<1\,.
\end{equation}
\end{itemize}

\noindent In the case
$ab=-1$, let $\tau,\epsilon>0$ be sufficiently small such that $$b<-1/(a+\tau)<b+1 \text{ and } a-1<-1/(b-\epsilon)<a\,.$$ 
We have:
\begin{itemize}
\item[(i)] If $a<x<a+\tau$, then $b-1<\hat f(x)<-1/(a+\tau)$, and 
$|x\hat f(x)|\le |a/(a+\tau)|<1$. If $a+\tau\le x<0$, then $|x\hat f(x)|\le |b(a+\tau)|<1$.
\item[(ii)] If $b-\epsilon< x<b$, then $0<\hat f(x)<a+1$ and one has either \eqref{eq-rel1} with $n\ge 2$ or \eqref{eq-rel2}. If $0<x\le b-\epsilon$, then one has \eqref{eq-rel1} or \eqref{eq-rel2} where $b$ is replaced by $b-\epsilon$.
\end{itemize}

In conclusion, there exists a constant $\gamma>1$ such that for every $x\in \bigcap_{i=0}^K \hat f_{a,b}^{-i}(X)$ some iterate $\hat f^n_{a,b}(x)$ with $n\le K+1$ satisfies the condition $|(\hat f_{a,b}^n)'(x)|\ge \gamma$. This implies that the iterate $\hat f^N_{a,b}$, with $N=(K+1)!$, is uniformly expanding, i.e. it satisfies property (U). Since properties (A) and (F) are automatically satisfied by any iterate of $\hat f_{a,b}$ (see \cite{Zw}), we have that $\hat F^N_{a,b}$ is Bernoulli. Using one of Ornstein's results \cite[Theorem 4, p. 39]{O}, it follows that $\hat F_{a,b}$  is Bernoulli.
\end{proof}

The next result gives a formula of the measure theoretic entropy of $(\hat F_{a,b},\nu_{a,b})$.

\begin{thm}
The measure-theoretic entropy of  $(\hat F_{a,b},\nu_{a,b})$ is given by
\begin{equation}\label{F-entropy}
h_{\nu_{a,b}}(\hat F_{a,b})=\frac{1}{K_{a,b}}\frac{\pi^2}{3}
\end{equation}
\end{thm}
\begin{proof}
To compute the entropy of this two-dimensional map, we use Abramov's formula \cite{Ab}: 
$$h_{\tilde m}(\{\phi^t\})=\frac{h_{\rho_{a,b}}(R_{a,b})}{\int_{\Lambda_{a,b}}g_{a,b}d\rho_{a,b}}\,,$$
where $\tilde m$ is the normalized Liouville measure $d\tilde m=\frac{dm}{m(SM)}$. It is well-known that $m(SM)=\pi^2/3$ (see \cite{AF1}) and $h_{\tilde m}(\{\phi^t\})=1$ (see, e.g., \cite{Su1}).
The  measure $d\tilde m$ can be represented by the Ambrose-Kakutani theorem \cite{AmKa} as a smooth probability measure on the space $\Lambda_{a,b}^{g_{a,b}}$ 
\begin{equation}
d\tilde m=\frac{d\rho_{a,b}dt}{\int_{\Lambda_{a,b}}g_{a,b}d\rho_{a,b}}
\end{equation}
where $d\rho_{a,b}$ is the probability measure on the cross-section $\Lambda_{a,b}$ given by \eqref {drho}. This implies that
\[
d\tilde m=\frac{d\rho dt}{K_{a,b}\int_{\Lambda_{a,b}}g_{a,b}d\rho_{a,b}}=\frac{dm}{K_{a,b}\int_{\Lambda_{a,b}}g_{a,b}d\rho_{a,b}}\,.
\]
Therefore $K_{a,b}\int_{\Lambda_{a,b}}g_{a,b}d\rho_{a,b}=m(SM)={\pi^2}/{3}$ and
\[
h_{\nu_{a,b}}(\hat F_{a,b})=h_{\rho_{a,b}}(R_{a,b})=\int_{\Lambda_{a,b}}g_{a,b}d\rho_{a,b}=\frac{1}{K_{a,b}}\frac{\pi^2}{3}\,.
\]
\end{proof}

Since $(\hat F_{a,b},\nu_{a,b})$ is the natural extension of $(\hat f_{a,b},\mu_{a,b})$, the measure-theoretic entropies of the two systems coincide, hence
\begin{equation}\label{fab-entropy}
h_{\mu_{a,b}}(\hat f_{a,b})=\frac{1}{K_{a,b}}\frac{\pi^2}{3}\,.
\end{equation}

As an immediate consequence of the above entropy formula we derive a growth rate relation for the denominators of the partial quotients $p_n/q_n$ of $(a,b)$-continued fraction expansions, similar to the classical cases.

\begin{prop}
Let $\{q_n(x)\}$ be the sequence of the denominators of the partial quotients $p_n/q_n$ associated to the $(a,b)$-continued fraction expansion of $x\in [a,b)$. Then
\begin{equation}
\lim_{n\ra \infty}\frac{\log q_n(x)}{n}=\frac{1}{2}h_{\mu_{a,b}}(\hat f_{a,b})=\frac{1}{K_{a,b}}\frac{\pi^2}{6}\, \text{ for a.e. } x.
\end{equation}
\end{prop}
\begin{proof}
The proof is similar to the classical case: using the Birkhoff's ergodic theorem one has
$$
\lim_{n\ra \infty}\frac{\log q_n(x)}{n}=-\int_a^b\log|x|d\mu_{a,b}\,.
$$
At the same time, Rokhlin's formula tells us that 
$$h_{\mu_{a,b}}(\hat f_{a,b})=\int_a^b\log|\hat f'_{a,b}|d\mu_{a,b}=-2\int_a^b\log|x|d\mu_{a,b}\,,$$
hence the conclusion.
\end{proof}

\section{Some explicit formulas for the invariant measure $\mu_{a,b}$}
In order to obtain explicit formulas for $\mu_{a,b}$ and $h_{\mu_{a,b}}(\hat f_{a,b})$, one obviously needs an explicit description of the domain $D_{a,b}$. 
In \cite{KU3} we describe an algorithmic approach for finding the boundaries of $D_{a,b}$ for all parameter pairs $(a,b)$ outside of a negligible exceptional parameter set $\mathcal E$. Let us point out that the set $D_{a,b}$ may have an arbitrary large number of horizontal boundary segments. The qualitative structure of $D_{a,b}$ is given by the cycle properties of $a$ and $b$. This structure remains unchanged for all pairs $(a,b)$ having cycles with similar combinatorial complexity. For a large part of the parameter set the cycle descriptions are relatively simple (see \cite[Section 4]{KU3}) and we discuss it herein. 

In what follows, we focus our attention on the situation $-1\le a\le 0\le b\le 1$, and due to the symmetry of the parameter set with respect to the parameter line $a=-b$ we assume that $a\le -b$.  

We treat the case $1\le -\frac{1}{a}\le b+1$ and $a\le -\frac{1}{b}+m\le a+1$ (for some $m\ge 1$).
The coordinates of the corners of the boundary segments in the upper region $D_{a,b} \cap \{(u,w)| u<0, a\le w \le b\}$ are given by
$$
(-2,b-1), \Bigl(-\frac{3}{2},T^{-2}S(b-1)\Bigr), \dots,\Bigl(-\frac{m+1}{m},{(T^{-2}S)^{(m-1)}(b-1)}\Bigr),\Bigl(-1,-\frac{1}{a}-1\Bigr)
$$
while the corners of the boundary segments in the lower region $D_{a,b} \cap \{(u,w)| u>0, a\le w \le b\}$ are given by 
$$\Bigl(m,-\frac{1}{b}+m\Bigr), (m+1, a+1)\,.$$ Therefore the set $\hat \Lambda_{a,b}$ is given by
\begin{equation}
\begin{split}
\hat \Lambda_{a,b}= &\bigcup_{p=1}^{m-1}
[(T^{-2}S)^{p-1}(b-1),(T^{-2}S)^{p}(b-1)]\times[0,\frac{p}{p+1}]\\& \cup [(T^{-2}S)^{m-1}(b-1),-\frac{1}{a}-1]\times [0,\frac{m}{m+1}]\cup [-\frac{1}{a}-1,b]\times [0,1]\\
& \cup [a,-\frac{1}{b}+m]\times[-\frac{1}{m},0]\cup [-\frac{1}{b}+m,a+1]\times[-\frac{1}{m+1},0]
\end{split}
\end{equation}

\begin{figure}[htb]
\includegraphics{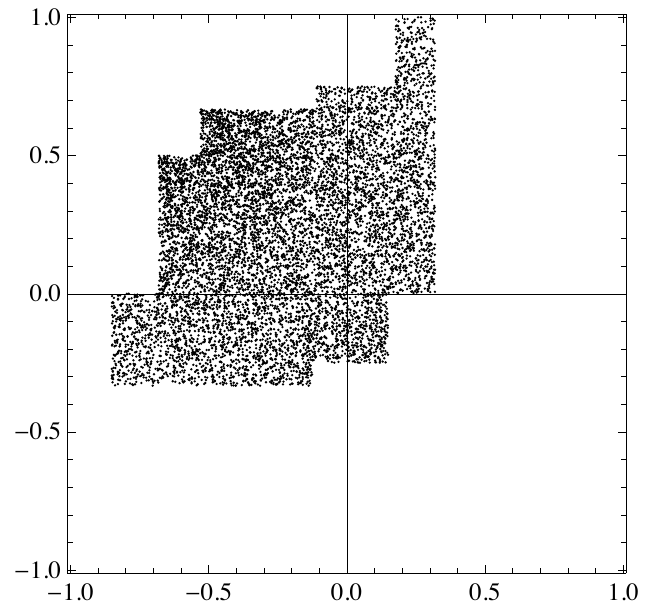}
\caption{Typical domain $\hat \Lambda_{a,b}$ for the case studied}
\end{figure}

\begin{thm} If $1\le -\frac{1}{a}\le b+1$ and $a\le -\frac{1}{b}+m\le a+1$, then
$$\mu_{a,b}=\frac{1}{K_{a,b}}h_{a,b}(x)dx\,,$$
where $K_{a,b}=\log[(m-a)(1+b)^{2-m}]$ and $h_{a,b}(x)=h^{+}_{a,b}(x)+h^{-}_{a,b}(x)$ with 
\begin{equation*}
h^+_{a,b}(x)=
\begin{cases}
\ds\frac{1}{x+\frac{p+1}{p}} &\text{ if } (T^{-2}S)^{p-1}(b-1)\le x < (T^{-2}S)^{p}(b-1), p=1,\dots, m-1\\[0.2in]
\ds\frac{1}{x+\frac{m+1}{m}} &\text{ if } (T^{-2}S)^{m-1}(b-1)\le x<-\frac{1}{a}-1\\[0.2in]
\ds\frac{1}{x+1} &\text{ if }   -\frac{1}{a}-1\le x<b
\end{cases}
\end{equation*}
and
\begin{equation*}
h^-_{a,b}(x)=\begin{cases}
\ds\frac{1}{m-x} &\text{ if }  a\le x <-\frac{1}{b}+m\\[0.2in]
\ds\frac{1}{m+1-x} &\text{ if }  -\frac{1}{b}+m\le x<a+1\,.
\end{cases}
\end{equation*}
\end{thm}

\begin{proof}
The density formulas are obtained from the simple integration result
\begin{equation}\label{intcd}
\int_{c}^{d}\frac{1}{(1+xy)^2}dy=-\frac{1}{x}\left(\frac{1}{1+dx}-\frac{1}{1+cx}\right)=\frac{d}{1+dx}-\frac{c}{1+cx}\,.
\end{equation}
For the density in the upper part of $\hat\Lambda_{a,b}$, $y\ge 0$, all integrals have the lower boundary $c=0$, hence the result of \eqref{intcd} becomes $1/(x+1/d)$.
This gives the description of $h^+_{a,b}(x)$. For the density in the lower part of $\hat\Lambda_{a,b}$, $y\le 0$, all integrals have the upper boundary $d=0$, hence the result $-1/(-1/c-x)$
and the description of $h^-_{a,b}(x)$.
By  a somewhat tedious computation, we get 
$$
K_{a,b}=\int_{\Lambda_{a,b}}h_{a,b}(x)dx=\log[(m-a)(1+b)^{2-m}]\,,
$$
and this completes the proof.
\end{proof}

\end{document}